    \documentclass[twoside,reqno,11pt]{fcaa-var} %

\usepackage{graphicx}
\usepackage{epsfig}

\usepackage{amsthm}
\usepackage{amsmath}
\usepackage{latexsym}
\usepackage{amsfonts}
\usepackage{amssymb}

\newtheoremstyle{theorem}
  {15pt}          
  {15pt}  
  {\sl}  
  {\parindent}
  {\sc}  
  {. }   
  { }    
  {}     
\theoremstyle{theorem}

\newtheorem{theorem}{Theorem}[section]

\newtheoremstyle{defi}
  {15pt}          
  {15pt}  
  {\rm}  
  {\parindent}     
  {\sc}  
  {. }    
  { }    
  {}     
\theoremstyle{defi}

 \def\proofname{\indent\rm P\ r\ o\ o\ f.}
 
 \def\qed{\hfill$\Box$} 

    \usepackage{hyperref} 
    \def\theequation{\arabic{section}.\arabic{equation}}

 \def\themycopyrightfootnote{\vspace*{3pt}
 \copyright \,Please cite this paper as : \\Fract. Calc. Appl. Anal., Vol. 21, No 2 (2018), pp. 486--508,
 DOI: 10.1515/fca-2018-0027
   \hfill  \vspace*{-36pt}
    }

 \title[EXTENSION OF MIKHLIN MULTIPLIER THEOREM \dots]
       {EXTENSION OF MIKHLIN MULTIPLIER THEOREM \\ [3pt] TO FRACTIONAL DERIVATIVES AND  \\ [3pt] STABLE PROCESSES}
 \author[\normalsize D. Karl\i]{\normalsize Deniz Karl\i } 

 \usepackage{amssymb}

 \newtheorem{lem}[theorem]{Lemma}

 \renewcommand{\P}{\mathbb{P}}
 \newcommand{\E}{\mathbb{E}}
 \newcommand{\Rd}{\mathbb{R}^d}
 \newcommand{\R}{\mathbb{R}}
 
 \newcommand{\F}{\mathcal{F}}
 \newcommand{\Lgen}{\mathcal{L}}
 \newcommand{\barr}{\begin{array}{rcl}}
 	\newcommand{\earr}{\end{array}}

 \newcommand{\disp}{\displaystyle}

 \newcommand{\ind}{1{\hskip -2.5 pt}\hbox{\textnormal I}}
 
 \DeclareMathOperator{\sign}{sign}

\begin{document}

 \vbox to 2.5cm { \vfill }


 \bigskip \medskip

 \begin{abstract}

In this paper, we prove a new generalized Mikhlin multiplier theorem whose conditions are given with respect to fractional derivatives in integral forms with two different integration intervals. We also discuss the connection between fractional derivatives and stable processes and prove a version of Mikhlin theorem under a condition given in terms of the infinitesimal generator of symmetric stable process. The classical Mikhlin theorem is shown to be a corollary of this new generalized version in this paper.

 \medskip


 \smallskip


 \end{abstract}

 \maketitle

 \vspace*{-16pt}



\section{Introduction and preliminaries}\label{sec:1} 

\setcounter{section}{1}
\setcounter{equation}{0}\setcounter{theorem}{0}

Fractional calculus is a rapidly growing area of Mathematics which lies between probability, differential equations and mathematical physics. It provides tools to analyze anomalous particle diffusion models which differ from the classical diffusions. Classical diffusions can be modeled as limits of random walks which correspond to Brownian motion.

In recent years, there has been a growing interest in general Levy processes which contain the most well-known process Brownian motion as a special case. Among Levy processes, stable processes have a special place due to their pleasing properties. Although stable processes do not share many nice properties of Brownian motion, they are the next best type of processes to consider. Hence it is a good starting point if we want to learn more about Levy processes and their applications. We will discuss a specific type of stable process in Section \ref{stochastic_section}. And we refer to \cite{applebaum}, \cite{bass_sing_int}, \cite{kolok} and \cite{ross} for an intensive study of the general theory.

Brownian motion, which is the main focus of classical analysis, has a close relation to classical (non-fractional) calculus. The well-known generator of Brownian motion is given $\frac{1}{2}\Delta$ where $\Delta$ is the Laplacian operator.

Moreover, it is known that the transition density, $p(t,x)$, solves the diffusion equation
\begin{align*}
\frac{\partial p}{\partial t} = D\, \frac{\partial^2 p}{\partial x^2}
\end{align*}
where the Fourier transform of the density is $$\F(p(t,\cdot))(x)=e^{-\frac{1}{2}t\sigma^2x^2}$$ with $D=\sigma^2/2.$

Following this diffusion equation, another equation
\begin{align*}
\frac{\partial p}{\partial t} = D\, \frac{\partial^\alpha p}{\partial x^\alpha}
\end{align*}
attracts attention where $ \frac{\partial^\alpha }{\partial x^\alpha}$ is the derivative of fractional order. Especially, if $\alpha \in (0,2)$ this differential equation has some connection to (one- and two-sided) stable processes. In particular, it is related to symmetric (rotationally invariant) stable process whose infinitesimal generator is
\begin{align*}
\Lgen[f](x)=\frac{\alpha(\alpha-1)}{\Gamma(2-\alpha)}\int_{\Rd-\{0\}} (f(y+x)-f(x) -\ind_{\{|y|<1\}} y\cdot \nabla f ) \frac{dy}{|y|^{d+\alpha}}
\end{align*}
for $f\in Dom(\Lgen)$. (See \cite[P.162]{applebaum}.) If $\alpha>1$ then the third term of the integrand is needed for the convergence of the integral. However if $\alpha<1$, as in this paper, then this term cancels out due to the fact that its integral equals zero. Hence the infinitesimal generator becomes
\begin{align}\label{gen_of_stable}
\Lgen[f](x)=\frac{\alpha}{\Gamma(1-\alpha)}\int_{\Rd-\{0\}} (f(y+x)-f(x)) \frac{dy}{|y|^{d+\alpha}}
\end{align}
which will be very much alike to the definition the fractional derivative below. For more details on this relation we refer to \cite[Chapter 3]{frac_calc}. Throughout this paper, the parameter $\alpha$ is restricted to the case $\alpha \in (0,1)$ unless stated otherwise and $\Gamma(\cdot)$ above represents the usual gamma function.

One of the corner stones of the classical operator theory is the famous Mikhlin multiplier theorem. A multiplier $m$ is a function which is the Fourier transform of the kernel of a convolution operator. That is, if $T_m$ is an operator with kernel $\kappa$ so that $T_mf=f*\kappa$, then $\F(\kappa)=m$. Multiplier theorems allow one to study convolution operators of type $T_m$ through their multipliers $m$ which provide useful simplifications. Mikhlin's theorem provides a control on the convolution operator if the corresponding multiplier obeys some growth condition. The most general version of this theorem is stated in dimensions $d\geq 1$. Since our focus in this paper is dimension one, we state the argument in $d=1$ only. (See \cite[Theorem 5.2.7]{grafakos} for $d\geq 1$.) For a study from a purely analytic point, one may see \cite{grafakos}, \cite{stein1}, \cite{stein2} and \cite{stein3}, and for a study from a probabilistic point, we refer \cite{karli_1} and \cite{karli_2}.

\begin{theorem}[Mikhlin]\label{classic_mikhlin}
	Suppose $d=1$, $c>0$ and $m$ is a bounded differentiable function on $\R-\{0\}$ so that $|m'(x)|\leq c/|x|$, $x\not= 0$. Then the convolution operator $T_mf=f*\kappa$, where $m$ is the Fourier transform of $\kappa$, can be extended from a bounded operator on $L^p\cap L^2$ to a bounded operator on $L^p(\R)$ for $p\in (1,\infty)$. 	
\end{theorem}

After the developments in \cite{karli_1,karli_2,karli_3}, we have many fundamental tools to work in the case of symmetric stable processes. These tools open a door to a wide range of new applications. One should be able to relax and extend restrictions of the classical theorems using these acquired tools.

In this paper, we focus on the extension of classical Mikhlin multiplier theorem, Theorem \ref{classic_mikhlin}. Our aim is to use fractional derivative and our process, which is a product of a symmetric stable process with a one dimensional Brownian motion, to obtain a general Mikhlin multiplier theorem which presents the statement \ref{classic_mikhlin} as a corollary.

For the rest of this paper, we assume that we are working in dimension one only. The notation $\F(f)$ is reserved for the Fourier transform of an appropriate function $f$, that is,
$$\F(f)=\frac{1}{\sqrt{2\pi}}\int_\R f(y) e^{-ixy} dy.$$

\subsection{Basics on fractional derivatives}

 One approach to define fractional derivative operator is by means of its Fourier transform. It is well-known that for the classical derivative operator $$\F\left(\frac{d^n}{dx^n}f\right)(x)=(ix)^n \F(f)(x).$$ Using this relation, we define the operator $d^\alpha/dx^\alpha$ as the operator for which
\begin{align}\label{fourier_frac_der}
\F\left(\frac{d^\alpha}{dx^\alpha}f\right)(x)=(ix)^\alpha \F(f)(x).
\end{align}
(See \cite[Chapter 2]{frac_calc} for details.) Two widely used definitions of this operator are known as Caputo and Riemann-Liouville forms. Both of these two forms can be obtained from another form of fractional derivative, which is called the generator form (see \cite[page 30]{frac_calc}). In this paper we consider this form whenever we refer fractional derivative. In order to fix the notation we will use $D_\alpha$ to denote the (positive) fractional derivative in generator form. This operator is given by
\begin{align}\label{frac_der_def}
D_\alpha[f](x)=\frac{\alpha}{\Gamma(1-\alpha)}\int_{0}^{\infty} \left( f(x)- f(x-y)\right) \frac{dy}{y^{1+\alpha}}.
\end{align}
with the domain $Dom(D_\alpha)$. Note that this domain includes the set of all bounded functions with continuous bounded first order derivatives, and hence , in particular, compactly supported continuous functions with continuous first order derivatives

In the literature of fractional derivatives, there is a second form, called the negative fractional derivative. It is almost the same integral definition as above except the fact that the domain of the integral is taken to be $(-\infty,0)$. After a basic change of variables, one can define negative fractional derivative in generator form by
\begin{align}
D^-_\alpha[f](x)=\frac{\alpha}{\Gamma(1-\alpha)}\int_{0}^{\infty} \left( f(x)- f(x+y)\right) \frac{dy}{y^{1+\alpha}}
\end{align}
with the domain $Dom(D^-_\alpha)$. These forms are related to one-sided stable processes. The correspondence is between stable processes with positive jumps and positive fractional derivative and between  stable processes with negative jumps and negative fractional derivative. Reader may find the details of this relation in \cite[Chapter 2 and 3]{frac_calc}. Later we will discuss the generator of two-sided symmetric stable processes as well.

Since both Caputo and Riemann-Liouville forms can be obtained by integration by parts applied to (\ref{frac_der_def}), the difference between them seems to be due to the boundary values. Hence by using the generator form, we eliminate this vague point in the definition.

Let us focus on $D_\alpha$  now. We need to check if  Fourier transform of this operator is $(i x)^\alpha$ as desired in  (\ref{fourier_frac_der}). By carrying out the necessary calculations,
\begin{align*}
\F\left(D_\alpha[f]\right)(x) &=\int_{\R } D_\alpha[f](y) e^{-ixy} dy = \frac{\alpha}{\Gamma(1-\alpha)}  \F(f) \int_0^\infty \left( 1-e^{-ixy}\right) \frac{dy}{y^{1+\alpha}}
\end{align*}
where the last integral converges and equals $(ix)^\alpha\Gamma(1-\alpha)/\alpha $. So we have
\begin{align}\label{four_of_Da}
\F\left(D_\alpha[f]\right)(x) &= (ix)^\alpha \F(f).
\end{align}
Similarly, one can compute the Fourier transform of $D^-_\alpha$, which is
\begin{align}\label{four_of_Da_neg}
\F\left(D^-_\alpha[f]\right)(x) &= (-ix)^\alpha \F(f).
\end{align}

Fractional derivative does not share some properties which are satisfied by the classical derivative operator, such as the product (Leibniz) rule. However, we will need this type property for the sake of the proof of the main theorem. Hence we define an alternative relation which is close to the  product rule and call it extended product rule for fractional derivatives. For this purpose, we define two new operators $\Lambda_\alpha$ and $\Lambda^-_\alpha$ as follows: Let $f$ and $g$ be two real valued functions such that $g\in Dom(D_\alpha)$ and $f$ is bounded. Then define
\begin{align}\label{carre2}
\Lambda_\alpha[f,g](x)=\frac{\alpha}{\Gamma(1-\alpha)}\int_{0}^{\infty}(f(x)-f(x-y))(g(x)-g(x-y)) \frac{dy}{y^{1+\alpha}}
\end{align}
 for $x\in \R$. Similarly,  if $g\in Dom(D^-_\alpha)$ and $f$ is bounded then we define
\begin{align}\label{carre3}
\Lambda^-_\alpha[f,g](x)=\frac{\alpha}{\Gamma(1-\alpha)}\int_{0}^{\infty}(f(x)-f(x+y))(g(x)-g(x+y)) \frac{dy}{y^{1+\alpha}}
\end{align}
for $x\in \R$. Note that these operators are well-defined for the given conditions on $f$ and $g$. These operators are in close relation with the {\it Carr\'e de Champ} operator in probability theory (See \cite[page 5]{karli_1}.) This close relation encourages the following lemma.

\begin{lem}[Extended Product Rule]\label{extended_product_rule} {\it Suppose $f$ is a bounded function.
	\begin{itemize}
		\item[i. ] If $f,g \in Dom(D_\alpha)$ then
		\begin{align}
		D_\alpha[f \cdot g]=f\, D_\alpha[g]+g\, D_\alpha[f]-\Lambda_\alpha[f,g].
		\end{align}
		\item[ii. ] If $f,g \in Dom(D^-_\alpha)$ then
		\begin{align}
		D^-_\alpha[f \cdot g]=f\, D^-_\alpha[g]+g\, D^-_\alpha[f]-\Lambda^-_\alpha[f,g].
		\end{align}			
	\end{itemize}}
\end{lem}
\begin{proof}
	We prove only the first part of the lemma. The second part is almost identical to the first part.
	
	Let $f$ and $g$ be as given in the hypothesis. Then we have
	\begin{align*}
	& \alpha^{-1}	\Gamma(1-\alpha) \left[ D_\alpha[f\cdot g](x)- f(x)D_\alpha[g](x)-g(x)D_\alpha[f](x) \right]\\
	&\quad=\int_\R (f(x)g(x)-f(x-y)g(x-y)) \frac{dy}{y^{1+\alpha}}\\
	&\quad\qquad \qquad -\int_\R (f(x)g(x)-f(x)g(x-y)) \frac{dy}{y^{1+\alpha}}  \\
	& \quad\qquad \qquad \qquad \qquad -\int_\R (f(x)g(x)-f(x-y)g(x)) \frac{dy}{y^{1+\alpha}}\\
	&\quad= \int_\R -f(x)(g(x)-g(x-y))+f(x-y)(g(x)-g(x-y)) \frac{dy}{y^{1+\alpha}}\\
	&\quad=- \alpha^{-1}	\Gamma(1-\alpha)\Lambda_\alpha[f,g](x).
	\end{align*}
	Since $\Lambda_\alpha[f,g](x)$ is well-defined for $f$ being bounded and $g\in Dom(D_\alpha)$, so is $D_\alpha[f \cdot g]$ with the additional condition $f\in Dom(D_\alpha)$. Hence we obtain the desired result.
\end{proof}

\subsection{Stochastic background and previous results}\label{stochastic_section}

As we mentioned in the beginning of the section, there is close relation between fractional derivative and stable processes. In papers \cite{karli_1,karli_2,karli_3}, we studied the stochastic point of view of this theory in terms of a specific Levy process. Here we recall some fundamental properties and crucial results obtained in \cite{karli_1,karli_2,karli_3} to make this paper as self-sufficient as possible. For the details of the discussions we refer to the mentioned papers.

Our main focus in these papers was a product process $X_t$ with components $Y_t$ and $Z_t$, where $Y_t$ is a d-dimensional symmetric $\alpha$-stable process ($0< \alpha <2$) and $Z_t$ is a 1-dimensional Brownian motion. For this process, we built the fundamental tools of the theory in \cite{karli_1} and defined harmonic extension of a given $L^p(\Rd)$ function to the upper-half space $\Rd\times\R^+$.  By harmonic function, we mean a function which is harmonic with respect to the process $X_t$. Later, based on these harmonic extensions, we defined Littlewood-Paley functions which play a crucial role in the Potential Theory. The key ideas of the structure are as follows. Let $f\in L^p(\Rd)$, $(Y_t,\P^x)$ be a d-dimensional symmetric stable process  started at the point $x\in \Rd$ with the probability measure $\P^x$, $(Z_t,P^a)$ be a 1-dimensional Brownian motion started at the point $a\in \R^+$ with the probability measure $\P^a$. We define the product probability measure $\P^{(x,a)}=\P^x\times \P^a$ for the product process $X_t=(Y_t,Z_t)$ and the stopping time $T_0=\inf\{s \geq 0 : Z_s=0\}$ which is the first hitting time of $X_s$ to the boundary $\Rd\times\{0\}$. We denote expectations with respect to probability measures $\P^a, \P^x$ and $\P^{(x,a)}$ by $\E^a, \E^x$ and $\E^{(x,a)}$, respectively. Then we define the harmonic extension of $f$ to the upper half-space by
\begin{align}\label{harm_ext}
f(x,t)=Q_tf(x)=\E^{(x,a)}(f(X_{T_0}))=\int_0^\infty \E^x(f(Y_s))\P^t(T_0\in ds)
\end{align}
for $(x,t) \in \Rd\times \R^+$. Here we denote both the function and its harmonic extension by the same letter $f$. This $Q_t$ is a convolution semi-group with the kernel $q_t$. So $Q_t$ satisfies the semi-group properties
\begin{align}\label{semigroup_prop}
Q_tQ_s=Q_{t+s}  \qquad \mbox{and} \qquad Q_0=\mbox{identity}.
\end{align}
To see the convolution kernel explicitly, we write
\begin{align*}
Q_tf(x)=f*q_t(x)= \int_{0}^{\infty} \int_{\Rd} f(y) p(s,x,y) dy \, \mu_t(ds),
\end{align*}
with
$$q_t(x)=\int_0^\infty p(s,x,0) \mu_t(ds).$$
Here $p(s,x,y)$ is the transition density of a d-dimensional symmetric $\alpha$-stable process with the Fourier transform
\begin{align}\label{pt_four}
\F(p(t,\cdot,0))(x)=e^{-t|x|^{\alpha}},
\end{align}
and $\mu_t$ is the exit distribution of a 1-dimensional Brownian motion from the domain $[0,\infty)$ with the explicit form
\begin{align*}
\mu_t(ds)= \frac{t}{2\sqrt{\pi}} e^{-t^2/(4s)} s^{-3/2} ds.
\end{align*}
The kernel $q_t$ is a probability kernel  with  the Fourier transform
\begin{align}\label{qt_fourier}
\F(q_t(\cdot))(x)=e^{-t|x|^{\alpha/2}},
\end{align}
which is due to (\ref{pt_four}) and Fubini's theorem since $\int_\R q_t(x)dx=1$.

By means of harmonic extensions, we define the Littlewood-Paley functions for the process $X_t$. The vertical, horizontal and general Littlewood-Paley functions are

\begin{align*}
\disp G^{\uparrow}_f(x) & =\disp\left[ \int_0^\infty t\, \left[ \partial_t Q_t f(x) \right]^2 \,dt  \right]^{1/2} , \\
\overrightarrow{G}_{f,\alpha}(x)  &  =\disp\left[ \int_0^\infty t\, \int_{\{|h|< t^{2/\alpha}\}} (Q_tf(x+h)-Q_tf(x))^2 \, \frac{dh}{|h|^{d+\alpha}} \,dt  \right]^{1/2} , \\
G_{f,\alpha}(x) &=\disp\left[\left(G^{\uparrow}_f(x)\right)^2 + \left(\overrightarrow{G}_{f,\alpha}(x)\right)  \right]^{1/2},
\end{align*}
respectively, for $x\in\Rd$. Then \cite[Theorem 7]{karli_1} together with P.A.Meyer's earlier result (see \cite[section 5]{karli_1}) shows that
\begin{align}\label{g_ineq_1}
\|{f}\|_{L^p(\Rd)} \leq c\, \|{G^{\uparrow}_f}\|_{L^p(\Rd)} \leq c\, \|{G_{f,\alpha}}\|_{L^p(\Rd)} \leq c\, \|{f}\|_{L^p(\Rd)}
\end{align}
for $p >1.$ This result is one of the main  accomplishments of \cite{karli_1}.

In \cite{karli_2}, we continue studying the key functionals of Littlewood-Paley Theory in this new setup. We define these functionals in the case of the process $X_t$ and study their boundedness properties. One of these operators, needed for this paper, is $G^*$ operator. Define the horizontal, the vertical and the general $G^*$ operators for $\lambda>1$ as
\begin{align*}
\overrightarrow{G}_{\lambda,f}^*(x)&=\left[ \int_0^\infty t \cdot K_t^\lambda*\Gamma_\alpha(Q_tf,Q_tf)(x) \, dt\right]^{1/2} , \\
\disp {G}_{\lambda,f}^{*,\uparrow}(x)&=\left[ \int_0^\infty t \cdot K_t^\lambda*(\frac{\partial}{\partial t}Q_tf(\cdot))^2(x) \, dt\right]^{1/2}\\
\disp {G}_{\lambda,f}^{*}(x)&= \left[\left[\overrightarrow{G}_{\lambda,f}^{*}(x)\right]^2+ \left[{G}_{\lambda,f}^{*,\uparrow}(x)\right]^2\right]^{1/2}.
\end{align*}
respectively, where the kernel $K_t^\lambda$ is
\begin{align*}
K_t^\lambda(x)=t^{-2d/\alpha}\left[ \frac{t^{2/\alpha}}{t^{2/\alpha}+|x|}\right]^{\lambda d},\qquad t>0.
\end{align*}
Here $\Gamma_\alpha(\cdot, \cdot)$ is the Carr\'e de Champ operator which is defined similar to (\ref{carre2}). Explicitly,
$$\Gamma_\alpha(Q_tf,Q_tf)(x)=\int_{\Rd} (Q_tf(x+y)-Q_tf(x))^2 \frac{dy}{|y|^{d+\alpha}}.$$
(See \cite[Proposition 1]{karli_1}.)

Moreover, \cite[Theorem 2.4]{karli_2} states that for $p\geq 2$ and $\lambda >1$
\begin{align}\label{g_ineq_2}
\|{{G}_{\lambda,f}^{*,\uparrow}}\|_{L^p(\Rd)}\leq c	\|{{G}_{\lambda,f}^{*}}\|_{L^p(\Rd)}\leq c \|{f}\|_{L^p(\Rd)}.
\end{align}
Combining (\ref{g_ineq_1}) and (\ref{g_ineq_2}), we have the following tool to study boundedness of an operator. If one can prove for a convolution operator $Tf=f*\kappa$ that
\begin{align*}
G^{\uparrow}_{Tf}(x) \leq c \,	{G}_{\lambda,f}^{*,\uparrow}(x)
\end{align*}
holds for almost every $x\in \Rd$, then (\ref{g_ineq_1}) and (\ref{g_ineq_2}) imply
\begin{align*}
\|{Tf}\|_{L^p(\Rd)} \leq c\,  \|{G^{\uparrow}_{Tf}}\|_{L^p(\Rd)} \leq c \,	\|{{G}_{\lambda,f}^{*,\uparrow}}\|_{L^p(\Rd)} \leq c\, \|{f}\|_{L^p(\Rd)} 
\end{align*}
for $p\geq 2$ and $\lambda >1$. This will be the main idea of the proof of the main theorem (Theorem \ref{main_thm}).

Finally, we developed tools to study Fourier multipliers in \cite{karli_3}  and used them to prove a particular multiplier theorem (\cite[Theorem 3.1]{karli_3}). These methods and some intermediate steps allowed us to approach in a different way than the classical version. We develop more on these tools in this paper and prove the extension of classical Mikhlin multiplier theorem in the next section.

\section{Mikhlin multiplier theorem: Generalizations to fractional derivatives}\label{sec:2} 

\setcounter{section}{2}
\setcounter{equation}{0}\setcounter{theorem}{0}

In this section we discuss the main result of this paper. Throughout this section, the letter $c$ is reserved for positive constants whose value may differ from line to line and whose value depend only on the parameter $\alpha$.

Before we state the main theorem, we will prove a technical lemma below which is needed later. For this purpose, let us define the radial function $K(\cdot)$ as
\begin{align}\label{k_func}
K(x)=|x|^{\alpha/2} e^{-\frac{1}{2}|x|^{\alpha/2}} \qquad , x\in \R.
\end{align}


\begin{lem}\label{techLemma}
	{\it	The $L^2(\R)$-norm of the function
	\begin{align}\label{def_of_J}
	J(x)=\int_\R |K(x-y)-K(x)| \frac{dy}{|y|^{1+\alpha}}
	\end{align}
	is bounded whenever $\alpha\in (1/2,1)$.}
\end{lem}

\begin{proof}
	To prove boundedness of $L^2(\R)$-norm, we split the integral
	\begin{align}
	\|{J}\|_{L^2(\R)}^2	= \int_\R \left[ \int_\R |K(x-y)-K(x))| \frac{dy}{|y|^{1+\alpha}}  \right]^2 dx 
	\end{align}
	into 4 sub-integrals and call them
	\begin{align*}
	I_1 & =\int_\R \left[ \int_{|y|>1} |K(x-y)-K(x)| \frac{dy}{|y|^{1+\alpha}}  \right]^2 dx, \\
	I_2 & =\int_{|x|>2} \left[ \int_{|y|<1} |K(x-y)-K(x)| \frac{dy}{|y|^{1+\alpha}}  \right]^2 dx, \\
	I_3 & =\int_{|x|<2} \left[ \int_{|y|<|x|/2} |K(x-y)-K(x)| \frac{dy}{|y|^{1+\alpha}}  \right]^2 dx, \\
	I_4 & =\int_{|x|<2} \left[ \int_{|x|/2<|y|<1} |K(x-y)-K(x)| \frac{dy}{|y|^{1+\alpha}}  \right]^2 dx. \\
	\end{align*}	
	Then we have
	$$ \|{J}\|_{L^2(\R)}^2 \leq c\, (I_1+I_2+I_3+I_4).$$
	First, consider the integral $I_1$. Note that $(\alpha/2) \, |y|^{-1-\alpha}$ is a probability kernel on $(-\infty,-1)\cup(1,\infty)$. Hence, we can obtain from Jensen's Inequality
	\begin{align*}
	I_1 & \leq \frac{2}{\alpha} \int_\R  \int_{|y|>1} (K(x-y)-K(x))^2 \frac{dy}{|y|^{1+\alpha}}   dx \\
	&\leq \frac{4}{\alpha} \int_\R  \int_{|y|>1} K^2(x-y) \frac{dy}{|y|^{1+\alpha}}   dx +  \frac{4}{\alpha} \int_\R  \int_{|y|>1}^\infty K^2(x) \frac{dy}{|y|^{1+\alpha}}   dx,  \\
	&\leq   \frac{8}{\alpha} \int_\R   K^2(x) dx \int_{|y|>1} \frac{dy}{|y|^{1+\alpha}}   dx,
	\end{align*}
	since $\phi(x)=x^2$ is a convex function. In the last line, we used Tonelli's Theorem to interchange the order of integrals and then applied a change of variables. Since both of the last two integrals are finite, we have $I_1 \leq c.$	
	
	Second, consider the integral $I_2$. We will work with the derivative of $K(x)$ which is
	\begin{align}
	K'(x)=\frac{\alpha}{2} \sign(x) e^{-\frac{1}{2}|x|^{\alpha/2}} \left( |x|^{\alpha/2-1}-\frac{1}{2}|x|^{\alpha-1} \right)
	\end{align}
	for $x\not= 0$. Here, $\sign(x)=x/|x|$ for $x\not= 0$. Hence we have the bound
	\begin{align*}
	|K'(x)| \leq \frac{\alpha}{2} \, e^{-\frac{1}{2}|x|^{\alpha/2}} \left( |x|^{\alpha/2-1}+|x|^{\alpha-1} \right).
	\end{align*}
	In the domain of this integral, $|x|>2$ and $|y|<1$ and so we have  $|x-y| \geq |x|/2 \geq 1$. Then for any $\xi \in (x-y,x)$, we have $|\xi|\geq |x|/2 \geq 1$ and so
	\begin{align*}
	|K'(\xi)| \leq  e^{-\frac{1}{4}|x|^{\alpha/2}}.
	\end{align*}
	If we apply Mean Value Theorem to $K(x-y)-K(x)$, then we obtain
	\begin{align*}
	I_2 & \leq   \int_{|x|>2} \left[ \int_{|y|<1} e^{-\frac{1}{4}|x|^{\alpha/2}} \, |y| \, \frac{dy}{|y|^{1+\alpha}}  \right]^2 dx \\
		&   \leq  \int_{|x|>2} e^{-\frac{1}{2}|x|^{\alpha/2}} dx \left[ \int_{|y|<1}  \, |y|^{-\alpha} \, dy  \right]^2  = c, \\
	\end{align*}
	since $\alpha \in (1/2,1).$
	
	Next, consider the integral $I_3$. In its domain, $ |x|/2 \leq |x-y| \leq  3|x|/2 $, hence for any $\xi \in (x-y,x)$, we have $|x|/2 \leq |\xi| \leq 3$ and so
	\begin{align*}
	|K'(\xi))| \leq 3\alpha\, |x|^{\alpha/2 -1}.
	\end{align*}
	Then by Mean Value Theorem,
	\begin{align*}
	|K(x-y)-K(x)| \leq 3\alpha |x|^{\alpha/2 -1} |y|.
	\end{align*}
	Using this bound, we obtain
	\begin{align*}
	I_3 & \leq 9 \alpha^2 \int_{|x|<2} \left[ \int_{|y|<|x|/2} |x|^{\alpha/2 -1} |y|^{-\alpha} dy  \right]^2 dx, \\
	& = 9 \alpha^2  \int_{|x|<2} \left[ \int_{|y|<|x|/2} |x|^{\alpha/2 -1} |y|^{(1-\alpha)/2} \, |y|^{(-1-\alpha)/2}dy  \right]^2 dx, \\
	& \leq  9 \alpha^2  \int_{|x|<2} |x|^{-1} \left[ \int_{|y|<|x|/2} |y|^{(-1-\alpha)/2}dy  \right]^2 dx, \\
	& = c  \int_{|x|<2} |x|^{-\alpha}  dx. \\
	\end{align*}
	The last integral converges, since $\alpha < 1.$
	
	Finally, in the domain of the integral $I_4$, we have $|x| < 2|y|$ and so
	\begin{align*}
	|K(x-y)-K(x)|\leq |x-y|^{\alpha/2}+|x|^{\alpha/2} \leq 5 \, |y|^{\alpha/2}.
	\end{align*}
	Then we have
	\begin{align*}
	I_4  & \leq c  \int_{|x|<2} \left[ \int_{|x|/2<|y|<1}  \frac{dy}{y^{1+\alpha/2}}  \right]^2 dx \leq c \int_{|x|<2} |x|^{-\alpha} dx = c. \\
	\end{align*}
	Therefore, we have  $\|{J}\|_{L^2(\R)}\leq c$ where $c$ depends only on $\alpha$.
\end{proof}


Next, we need an upper bound  for the $L^2(\R)$-norm of the function $$(s^{2/\alpha}+|x|)^{1/\alpha} \, \partial_s Q_{s/2} \kappa(x),$$ for the convolution kernel $\kappa$.
Recall that $Q_s\kappa(x)$ is the harmonic extension of $\kappa(x)$ to $\R\times\R^+$ by (\ref{harm_ext}) with respect to the process $X_t=(Y_t,Z_t)$.
For this purpose, let us recall the (positive) fractional derivative $$D_\alpha[f](x)=\frac{\alpha}{\Gamma(1-\alpha)} \int_0^\infty (f(x)-f(x-y)) \frac{dy}{y^{1+\alpha}},$$
for $\alpha \in (0,1)$ and prove the following result.

\begin{theorem}\label{subintegral}
	Suppose  $\alpha \in (1/2,1)$, $m:\R\rightarrow\R$ is a bounded function with $m \in Dom(D_\alpha),$
	$$ \|m\|_\infty \leq C_1 \quad \mbox{and} \quad \left| D_\alpha[m](x) \right| \leq \frac{C_1}{|x|^\alpha} , \quad x\in \R-\{0\},$$
	for some $C_1 \in \R^+$, and $T_m$ is the convolution operator $T_mf=f*\kappa$ with $\F(\kappa)=m$. Then we have
	$$\|{\left( s^{2/\alpha} + |\cdot| \right)^{\alpha}  \partial_s Q_{s/2} \kappa(\cdot)}\|_{L^2(\R)} \leq c\, C_1 \, s^{1-1/\alpha} \qquad , s>0,$$
	for some $c>0$ depending only on $\alpha$.
\end{theorem}

\begin{proof}
	Let $s>0$ and consider the following partition of the square of the desired $L^2$-norm. Set
	\begin{align}\label{partition}
	I_1=& \int_{|x| \leq s^{2/\alpha}}  \left(s^{2/\alpha} + |x| \right)^{2\alpha} \left( \partial_s Q_{s/2} \kappa(x)\right)^2 dx,\\
	I_2=&\int_{|x|>s^{2/\alpha}}  \left(s^{2/\alpha} + |x| \right)^{2\alpha} \left( \partial_s Q_{s/2} \kappa(x)\right)^2 dx,\nonumber
	\end{align}
	so that
	$$\|{\left( s^{2/\alpha} + |\cdot| \right)^{\alpha}  \partial_s Q_{s/2} \kappa(\cdot)}\|_{L^2(\R)}^2 = I_1 + I_2.$$
	We note that $$I_1 \leq 2^{2\alpha}  \int_{|x| \leq s^{2/\alpha}}  s^{4} \left( \partial_s Q_{s/2} \kappa(x)\right)^2 dx.$$
	Moreover, we can bound $\partial_s Q_{s/2} \kappa(x)$ by $L^1$-norm of its Fourier transform. That is,
	\begin{align*}
	\left|\partial_s Q_{s/2}\kappa(x)\right| & \leq \int_\R \left| \F(\partial_s Q_{s/2}\kappa)(x)\right| dx,\\
	&=\frac{1}{2}\int_\R |x|^{\alpha/2} e^{-\frac{s}{2}|x|^{\alpha/2}} |m(x)| dx \\
	&\leq \frac{1}{2} \|m\|_\infty \int_\R |x|^{\alpha/2} e^{-\frac{s}{2}|x|^{\alpha/2}}  dx, \\
	&= \frac{1}{2} \|m\|_\infty\, s^{-1-2/\alpha}\int_\R |x|^{\alpha/2} e^{-\frac{1}{2}|x|^{\alpha/2}}  dx,
	\end{align*}
	by (\ref{qt_fourier}). Since the last integral converges, we have
	\begin{align*}
	I_1&\leq c\, C_1^2 \int_{|x|\leq s^{2/\alpha}} s^{2-4/\alpha} dx \leq c\, C_1^2\, s^{2-2/\alpha}.
	\end{align*}
	
	For the next part, let us define $K_s(x)=|x|^{\alpha/2} e^{-\frac{s}{2}|x|^{\alpha/2}}$ for $s>0$ and $x\in \R$. Here we note that $K_1(x)=K(x)$ where $K(x)$ is as defined in (\ref{k_func}). Then we have
	\begin{align*}
	I_2 &  \leq 2^{2\alpha}\, \int_{\R} |x|^{2\alpha} \left( \partial_s Q_{s/2} \kappa(x)\right)^2 dx
	= 2^{2\alpha} \,  \int_{\R}  \left( (ix)^\alpha \partial_s Q_{s/2} \kappa(x)\right)^2 dx.
	\end{align*}
	By Plancherel's identity, (\ref{four_of_Da}) and (\ref{qt_fourier}), the last integral equals
	\begin{align*}
	\int_{\R}  \left(D_\alpha [\F( \partial_s Q_{s/2} \kappa)](x)\right)^2 dx &=  \int_{\R}  \left(D_\alpha [|\cdot|^{\alpha/2} e^{-\frac{s}{2}|\cdot|^{\alpha/2}} \,m(\cdot) ](x)\right)^2 dx \\& =  \int_{\R}  \left(D_\alpha[ K_s \,m ](x)\right)^2 dx,
	\end{align*}
	which is bounded by 4 times the sum
	\begin{align*}
	    \int_{\R}  \left(\Lambda_\alpha [K_s, m ](x)\right)^2 dx &+ \int_{\R} K_s^2(x) \left(D_\alpha [  m ](x)\right)^2 dx   \\& \quad\quad+  \int_{\R} m^2(x) \left(D_\alpha [K_s ](x)\right)^2 dx
	\end{align*}
	due to the extended product rule (Lemma \ref{extended_product_rule}). By our assumptions on $m$, this sum is bounded by a constant multiple of
	\begin{align*}
	& \int_{\R} K_s^2(x) |x|^{-2\alpha} dx  +   \int_{\R}  \left(\int_0^\infty |K_s(x-y)-K_s(x)| \frac{dy}{y^{1+\alpha}}\right)^2 dx\\
	&=  s^{2-2/\alpha} \int_{\R}  K_1^2(x) |x|^{-2\alpha} dx  +  \int_{\R}  \left(\int_0^\infty |K_s(x-y)-K_s(x)| \frac{dy}{y^{1+\alpha}}\right)^2 dx,
	\end{align*}
	by using the scaling property
	\begin{align}\label{scaling_Ks}
	K_s(x)=s^{-1}K_1(s^{2/\alpha}x)=s^{-1}K(s^{2/\alpha}x).
	\end{align}
	Moreover, we have that $K_1(\cdot)|\cdot|^{-\alpha}\in L^2(\R)$, since $\alpha<1$. Hence it is enough to show that
	\begin{align*}
	I_3:=\int_{\R}  \left(\int_0^\infty |K_s(x-y)-K_s(x)| \frac{dy}{y^{1+\alpha}}\right)^2 dx \leq c \, s^{2-2/\alpha}
	\end{align*}
	to complete the proof.
	
	Using the scaling property (\ref{scaling_Ks}) of $K_s$, we have
	\begin{align*}
	I_3= c \, s^{-2} \int_{\R}  \left(\int_0^\infty |K(s^{2/\alpha}x-s^{2/\alpha}y)-K(s^{2/\alpha}x)| \frac{dy}{y^{1+\alpha}}\right)^2 dx .
	\end{align*}
	Now we apply two steps of change of variables with $z=s^{2/\alpha}x$ and $w=s^{2/\alpha}y$ to obtain
	\begin{align*}
	I_3= c \, s^{2-2/\alpha} \int_{\R}  \left(\int_0^\infty |K(z-w)-K(z)| \frac{dw}{w^{1+\alpha}}\right)^2 dz .
	\end{align*}
	The last integral converges by Lemma \ref{techLemma}. Then the result follows.	
\end{proof}

We are ready to state and prove the main theorem of this paper. The following theorem is a generalization of the classical Mikhlin multiplier theorem. First we prove the general statement below. Then we give reasoning why it is a more general result than the original one.


\begin{theorem}\label{main_thm}
	Suppose  $\alpha \in (1/2,1)$ and that $m:\R\rightarrow\R$ is a bounded function with $m\in Dom(D_\alpha)$, $$ \|m\|_\infty \leq C_1 \quad \mbox{and} \quad \left|  D_\alpha[m](x) \right| \leq \frac{C_1}{|x|^\alpha} , \quad x\in \R-\{0\},$$ for some $C_1>0$. Then $m$ is a Fourier multiplier and the corresponding convolution operator $T_m$ can be extended from $L^p \cap L^2$ to $L^p$ for $p\in(1,\infty)$.
\end{theorem}

\begin{proof}
	Let $p\geq 2$ and $f\in L^p(\R)$ be a compactly supported continuous function, that is, $f\in \mathcal{C}_c(\R)$. Let $T_m$ be the convolution operator with kernel $\kappa$ corresponding to m, that is, $\F(T_mf)=\F(f*\kappa)=\F(f)\cdot m$.
	
	We consider the vertical Littlewood-Paley function $G_{T_mf}^\uparrow(x)$. By its definition
	\begin{align*}
	\left( G_{T_mf}^\uparrow(x) \right)^2 = \int_{0}^{\infty} s \left( \partial_s Q_s T_mf(x)\right)^2 \,ds.
	\end{align*}
	We note that by the semi-group property (\ref{semigroup_prop}), we have $Q_s=Q_{s/2}Q_{s/2}$ and $q_s=q_{s/2}*q_{s/2}$, which leads to $\partial_sq_s=2q_{s/2}*\partial_sq_{s/2}$. Next,  we observe that $\partial_s Q_{s}T_mf(x)=2Q_{s/2}T_m(\partial_sQ_{s/2}f)(x)$ by means of  their Fourier transforms, that is,
	\begin{align*}
	& \F( 2Q_{s/2}T_m(\partial_sQ_{s/2}f))=2\F(q_{s/2})\,m\,\F(\partial_sq_{s/2})\F(f) \\
	&\qquad =\F(2q_{s/2}*\partial_sq_{s/2})\,m\,\F(f)=\F(\partial_sq_{s})\,m\,\F(f) =\F(\partial_s Q_{s}T_mf).
	\end{align*}
	Then
	\begin{align}\label{GEqn1}
	\left( G_{T_mf}^\uparrow(x) \right)^2 &=4\int_0^\infty s\, \left|Q_{s/2}T_m(\partial_sQ_{s/2}f)(x) \right|^2 \, ds.
	\end{align}
	Fix any $x\in\R$. Then
	\begin{align*}
	\left|Q_{s/2}T_m(\partial_sQ_{s/2}f)(x) \right| & \leq \int_{\R} \left| \F(Q_{s/2}T_m(\partial_sQ_{s/2}f))(y) \right|dy \\
	&  = \int_{\R} e^{-\frac{s}{2}|y|^{\alpha/2}} |m(y)| \, \, |y|^{\alpha/2} e^{-\frac{s}{2}|y|^{\alpha/2}} |\F(f) (y)| dy \\
	& \leq c\, C_1 \|{f}\|_{L^1(\R)} \int_{\R}K_{2s}(y) dy\\
	& \leq c \, C_1 s^{-1-2/\alpha} \|{K_1}\|_{L^1(\R)} \|{f}\|_{L^1(\R)} 
	\end{align*}
	by scaling (\ref{scaling_Ks}).
	Here, $\|{f}\|_{L^1(\R)}$ is finite since $f\in \mathcal{C}_c(\R)$.
	Then we can see that $Q_{s/2}T_m(\partial_sQ_{s/2}f)\rightarrow 0$ as $s\rightarrow \infty$. Hence the integral in (\ref{GEqn1}) becomes
	\begin{align*}
	4\int_0^\infty s\, \left|\int_s^\infty \frac{t}{t}\, \partial_t Q_{t/2}T_m(\partial_tQ_{t/2}f)(x)dt \right|^2 \, ds.
	\end{align*}
	If we apply the Cauchy-Schwartz inequality here, we obtain
	\begin{align*}
	&\left(G^\uparrow_{T_mf}(x)\right)^2 \\
	&\qquad\leq 4 \int_0^\infty s \left[ \int_s^\infty t^{-2}dt\right] \left[ \int_s^\infty t^2\, (\partial_t Q_{t/2}T_m(\partial_tQ_{t/2}f)(x))^2 dt \right]\, ds\\
	&\qquad=4\, \int_0^\infty   \int_s^\infty t^2\, (\partial_t Q_{t/2}T_m(\partial_tQ_{t/2}f)(x))^2 dt  ds\\
	&\qquad=4\, \int_0^\infty   t^3\, (\partial_t Q_{t/2}T_m(\partial_tQ_{t/2}f)(x))^2 dt  .\\
	\end{align*}
	Since $T_m$ is a convolution operator, the last integral equals
	\begin{align*}
	& \int_0^\infty   s^3\, \left(\int_\R \partial_s Q_{s/2}f(x-y) \cdot \partial_sQ_{s/2}\kappa(y) dy \right)^2 ds =\\
	&   \int_0^\infty   s^3 \left(\int_\R (s^{2/\alpha}+|y|)^{-\alpha} \partial_s Q_{s/2}f(x-y) (s^{2/\alpha}+|y|)^{\alpha} \partial_sQ_{s/2}\kappa(y) dy \right)^2 ds. \\
	\end{align*}
	If we apply Cauchy-Schwartz Inequality, then this integral is bounded by
	\begin{align*}
	&\int_0^\infty   s^3\, \left[ \int_\R (s^{2/\alpha}+|y|)^{-2\alpha} \left(\partial_s Q_{s/2}f(x-y) \right)^2 dy \right.\\
	&\qquad \qquad  \qquad  \qquad  \qquad \left. \cdot \int_{\R} (s^{2/\alpha}+|y|)^{2\alpha}  \left(\partial_sQ_{s/2}\kappa(y)\right)^2 dy \right]  \,ds. \\
	\end{align*}
	By Theorem \ref{subintegral}
	$$\int_{\R} (s^{2/\alpha}+|y|)^{2\alpha}  \left(\partial_sQ_{s/2}\kappa(y)\right)^2 dy \leq c \, C_1^2 s^{2-2/\alpha}.$$
	Hence $\left(G^\uparrow_{T_mf}(x)\right)^2$ is dominated by
	\begin{align*}
	 c\, C_1^2 \int_0^\infty   s\, \int_\R s^{-2/\alpha} \left(\frac{s^{2/\alpha}}{s^{2/\alpha}+|y|} \right)^{2\alpha} \left(\partial_s Q_{s/2}f(x-y) \right)^2 dy \, ds.
	\end{align*}
	Note that the last integral is the definition of the operator $\left(G_{\lambda, f}^{*,\uparrow} \right)^2$ if we take $\lambda=2\alpha$. Here $2\alpha>1$, and so $G_{\lambda, f}^{*,\uparrow} $ is well-defined. This implies that
	\begin{align}\label{ineq_part_1}
	G_{T_mf}^\uparrow(x) \leq c \,C_1\, G_{\lambda, f}^{*,\uparrow}
	\end{align}
	for any $x\in \R$.
	
	By (\ref{g_ineq_2})  (or \cite[Theorem 2.4]{karli_2}) ,
	\begin{align}\label{ineq_part_2}
	\|{ G_{\lambda, f}^{*,\uparrow} }\|_{L^p(\R)} \leq c \|{f}\|_{L^p(\R)}
	\end{align}
	for $f\in L^p(\R)$, $p\geq 2$ and $\lambda>1$. Moreover, by  (\ref{g_ineq_1})  (or \cite[Lemma 1.5]{karli_2}) , we also have
	\begin{align}\label{ineq_part_3}
	\|{ G_{f}^\uparrow }\|_{L^p(\R)} \geq c \|{f}\|_{L^p(\R)}
	\end{align}
	for $f\in L^p(\R)$, $p>1$.
	Hence by putting (\ref{ineq_part_1}), (\ref{ineq_part_2}) and (\ref{ineq_part_3}) together, we have
	\begin{align}\label{Gp_Ineq}
	\|{T_mf}\|_{L^p(\R)} \leq c\,	\|{G_{T_mf}^\uparrow}\|_{L^p(\R)} \leq c \, C_1\, \|{G_{\lambda, f}^{*,\uparrow}}\|_{L^p(\R)} \leq c\,C_1\, \|{f}\|_{L^p(\R)},
	\end{align}
	for $p\geq 2$. Since compactly supported continuous functions are dense in $L^p(\R)$, the inequality  $$	\|{T_mf}\|_{L^p(\R)} \leq c\,C_1	\|{f}\|_{L^p(\R)}$$  extends to all $f\in  L^p(\R)$ with $p\geq 2$.
	
	Finally, for the dual case $p\in (1,2)$, let $f\in L^p(\R)$, $q\in(2,\infty)$ such that $1/p+1/q=1$, and $g\in L^q(\R)$ be a continuous function with compact support. Denote the operator corresponding to the kernel $\tilde \kappa(x)=\kappa(-x)$ by $\tilde T_m$. Then by Fubini's Theorem,
	\begin{align*}
	\left| \int_{\R} T_mf(x)g(x)dx\right| = 	\left| \int_{\R} f(x)\tilde T_m g(x)dx\right| .
	\end{align*}
	By H\"older's Inequality and (\ref{Gp_Ineq}), the inner product above is less than
	\begin{align*}
	\|{f}\|_{L^p(\R)}\|{\tilde T_m g}\|_{L^q(\R)}\leq c\,C_1 \, \|{f}\|_{L^p(\R)}\|{ g}\|_{L^q(\R)}.
	\end{align*}
	Since compactly supported continuous functions are dense in $L^q(\R)$, this holds for all $g\in  L^q(\R)$ with $q \in (2,\infty)$. Hence we have
	$$\|{T_mf}\|_{L^p(\R)}\leq c \, C_1\,\|{f}\|_{L^p(\R)}$$	
	for any $f\in L^p(\R)$ and $p>1$. This completes the proof.
\end{proof}


To underline the importance of this result, we need to point the relation between this result and the classical version of the theorem.
Clearly, the classical version of Mikhlin multiplier theorem (Theorem \ref{classic_mikhlin}) is a corollary of Theorem \ref{main_thm}. To see this, let $m$ be a function satisfying conditions of Theorem \ref{classic_mikhlin}. Then we have
\begin{align*}
\left| D_\alpha[m](x) \right| \leq c\int_{0}^{|x|/2} |m'(\xi_{x,y})| \frac{dy}{y^{\alpha}} + c\int_{0}^{|x|/2}  \frac{dy}{y^{1+\alpha}}
\end{align*}
for some $\xi_{x,y}$ between $x$ and $x-y$ by Mean Value Theorem. Since $y<|x|/2$ for the first integral, we have $|x-y|\geq |x|/2$ and so $|\xi_{x,y}|\geq |x|/2$. Then
\begin{align*}
\left| D_\alpha[m](x) \right| & \leq c\int_{0}^{|x|/2} |\xi_{x,y}|^{-1} \frac{dy}{y^{\alpha}} + c|x|^{-\alpha}\\
&  \leq  c|x|^{-1} \int_{0}^{|x|/2}  \frac{dy}{y^{\alpha}} + c |x|^{-\alpha}  \leq c |x|^{-\alpha}.
\end{align*}
Hence $m \in Dom(D_\alpha)$ and it satisfies conditions of the extended Mikhlin multiplier Theorem (Theorem \ref{main_thm}). This shows that classical Mikhlin multipliers form a subclass of those which are characterized by Theorem \ref{main_thm}.

\smallskip

So far we established the desired result under conditions with respect to the positive fractional derivative $D_\alpha$. In the introductory section, we also defined the negative fractional derivative, which is studied in the literature as much as $D_\alpha$. So one may ask if the same result holds for the negative fractional derivative $D_\alpha^-$. We note that our technique does not rely on difference between $D_\alpha$ and $D_\alpha^-$. Hence it is very tempting to study the same problem for $D_\alpha^-$. Below we prove the statement of the Main Theorem under the conditions given with respect to $D_\alpha^-$. Since the proof shares some details from the lines of the proof for $D_\alpha$, we refer those points to keep ourself from repeating the same arguments. We mainly go over the proof of Theorem  \ref{subintegral} and outline the details which are needed. So we obtain the following result for $D_\alpha^-$.

\begin{theorem}\label{subintegral2}
	Suppose  $\alpha \in (1/2,1)$, $m:\R\rightarrow\R$ is a bounded function with $m\in Dom(D_\alpha^-)$, $$ \|m\|_\infty \leq C_1 \quad \mbox{and} \quad \left| D^-_\alpha[m](x) \right| \leq \frac{C_1}{|x|^\alpha} , \quad x\in \R-\{0\},$$ for some $C_1 \in \R^+$, and $T_m$ is the convolution operator $T_mf=f*\kappa$ with $\F(\kappa)=m$. Then we have
	\begin{align}\label{norm_est}
	\|{\left( s^{2/\alpha} + |\cdot| \right)^{\alpha}  \partial_s Q_{s/2} \kappa(\cdot)}\|_{L^2(\R)} \leq c\, C_1 \, s^{1-1/\alpha} \qquad , s>0,
	\end{align}
	for some $c>0$ depending only on $\alpha$.
\end{theorem}
\begin{proof}
	Consider the same partition $I_1$ and $I_2$ of the square of the norm (\ref{norm_est}) as in (\ref{partition}) and consider two integrals on the domains ${|x| \leq s^{2/\alpha}}$ and ${|x| > s^{2/\alpha}}$, respectively. The same argument as in the first part of the proof of Theorem  \ref{subintegral} shows that $$I_1\leq c\,C_1^2\,s^{2-2/\alpha}.$$
	For $I_2$, note that we have
	\begin{align*}
	I_2 &  \leq 2^{2\alpha}\, \int_{\R} |x|^{2\alpha} \left( \partial_s Q_{s/2} \kappa(x)\right)^2 dx  = 2^{2\alpha} \,  \int_{\R}  \left( (-ix)^\alpha \partial_s Q_{s/2} \kappa(x)\right)^2 dx\\
	& = 	\int_{\R}  \left(D^-_\alpha [\F( \partial_s Q_{s/2} \kappa)](x)\right)^2 dx    =  \int_{\R}  \left(D^-_\alpha [K_s \,m ](x)\right)^2 dx
	\end{align*}
	by (\ref{four_of_Da_neg}) and the Plancherel's identity. The last integral is bounded by 4 times the sum
	\begin{align*}
	&    \int_{\R}  \left(\Lambda^-_\alpha [K_s , m ](x)\right)^2 dx + \int_{\R} K_s^2(x) \left(D^-_\alpha [  m ](x)\right)^2 dx\\
	&\qquad\qquad +  \int_{\R} m^2(x) \left(D^-_\alpha [ K_s ](x)\right)^2 dx
	\end{align*}
	due to the extended product rule (Lemma \ref{extended_product_rule}). By the assumptions on $m$ and $D^-_\alpha [  m ]$, the last line is bounded by a constant multiple of
	\begin{align*}
	&  C_1^2 s^{2-2/\alpha} \int_{\R}  K_1^2(x)  dx  + C_1^2 \int_{\R}  \left(\int_0^\infty |K_s(x+y)-K_s(x)| \frac{dy}{y^{1+\alpha}}\right)^2 dx\\
	&\qquad\qquad\leq  C_1^2 s^{2-2/\alpha} \left( \|{K_1}\|_{L^2(\R)}^2 +  \|{J}\|_{L^2(\R)}^2 \right) \leq c\, C_1^2 s^{2-2/\alpha},
	\end{align*}
	where $J$ is defined in (\ref{def_of_J}).
\end{proof}


Since we have the same estimate (\ref{norm_est}) as in Theorem \ref{subintegral}, and (\ref{norm_est}) is all we need to prove the main theorem above, we obtain the same argument when we replace $D_\alpha$ by $D^-_\alpha$. Hence the proof of the following theorem follows from the lines of the proof of the main theorem.


\begin{theorem}\label{main_thm_2}
	Suppose  $\alpha \in (1/2,1)$ and that $m:\R\rightarrow\R$ is a bounded function with $m\in Dom(D_\alpha^-)$, $$ \|m\|_\infty \leq C_1 \quad \mbox{and} \quad \left|  D^-_\alpha[m](x) \right| \leq \frac{C_1}{|x|^\alpha} , \quad x\in \R-\{0\},$$ for some $C_1>0$. Then $m$ is a Fourier multiplier and the corresponding convolution operator $T_m$ can be extended from $L^p \cap L^2$ to $L^p$ for $p\in(1,\infty)$.
\end{theorem}


\section{Notes on the connection with symmetric stable processes } \label{sec:3} 
\setcounter{section}{3}
\setcounter{equation}{0}\setcounter{theorem}{0}

In this last section, we discuss similar results in terms of the infinitesimal generator $\Lgen$ given in (\ref{gen_of_stable}). The infinitesimal generator is given by means of semi-groups corresponding to  the process $Y_s$ as defined in the preliminary section. If $Y_s$ is a one dimensional symmetric stable process then $\{P_s\}_{s\geq 0}$, given by
\begin{align*}
P_s(f)(x)=\E^x(f(Y_s)),
\end{align*}
is a well defined a semi-group on the space of continuous functions vanishing at infinity. It satisfies two properties of semi-groups: $P_sP_t=P_{s+t}$ and $P_0f=f$ for any continuous function $f$ vanishing at infinity. The operator $\Lgen$ is then defined as the limit
\begin{align*}
\Lgen [f](x)=\lim\limits_{t\rightarrow 0}\frac{P_tf(x)-f(x)}{t}
\end{align*}
for any $x\in \R$ with the domain $Dom(\Lgen)$. Since its domain includes the continuous functions vanishing at infinity, this operator can be extended to the space $L^P(\R)$. Since a symmetric stable process is a Levy process, one can use the Levy representation to write this operator $\Lgen$ in the form of (\ref{gen_of_stable})

\begin{align*}
\Lgen[f](x)=\frac{\alpha}{\Gamma(1-\alpha)}\int_{\Rd-\{0\}} (f(y+x)-f(x)) \frac{dy}{|y|^{d+\alpha}}
\end{align*}
whenever $\alpha\in(0,1)$ with the Fourier transform $|x|^\alpha$.

Observe that both of the operators $D^-\alpha$ and $D_\alpha$ are actually the operator $\Lgen$ restricted to negative and positive real line, respectively. Hence we have the relation
\begin{align*}
\Lgen=-D_\alpha-D^-_\alpha
\end{align*}
with $Dom(\Lgen) \subseteq Dom(D_\alpha) \cap Dom(D_\alpha^-)$.

At this point, we ask the question if we can replace the condition on fractional derivative in Theorem \ref{main_thm} with a similar condition on $\Lgen$, that is, do we have the same conclusion if $\left|\Lgen[m](x)\right| \leq \frac{C_1}{|x|^\alpha}$ for $\alpha\in(0,1)$. We note that this cannot be concluded form either of Theorem \ref{main_thm} or Theorem \ref{main_thm_2}. However, the proof of both of these theorems rely on Theorem \ref{subintegral} and Theorem \ref{subintegral2}, which can be achieved for $\Lgen$ as well.

\begin{theorem}\label{subintegral3}
	Suppose  $\alpha \in (1/2,1)$, $m:\R\rightarrow\R$ is a bounded function with $m\in Dom(\Lgen)$, $$ \|m\|_\infty \leq C_1 \quad \mbox{and} \quad \left| \Lgen[m](x) \right| \leq \frac{C_1}{|x|^\alpha} , \quad x\in \R-\{0\},$$ for some $C_1 \in \R^+$, and $T_m$ is the convolution operator $T_mf=f*\kappa$ with $\F(\kappa)=m$. Then we have
		\begin{align}\label{norm_est_2}
	\|{\left( s^{2/\alpha} + |\cdot| \right)^{\alpha}  \partial_s Q_{s/2} \kappa(\cdot)}\|_{L^2(\R)} \leq c\, C_1 \, s^{1-1/\alpha} \qquad , s>0,
	\end{align}
	for some $c>0$ depending only on $\alpha$.
\end{theorem}
\begin{proof}
	We use the same partition $I_1$ and $I_2$ of the square of the norm as in (\ref{partition}). The first part of  the proof of Theorem  \ref{subintegral} shows that $$I_1\leq c\,C_1\,s^{2-2/\alpha}.$$
	For the next part, we have
	\begin{align*}
	I_2 &  \leq 2^{2\alpha}\, \int_{\R} |x|^{2\alpha} \left( \partial_s Q_{s/2} \kappa(x)\right)^2 dx  = 2^{2\alpha} \,  \int_{\R}  \left( |x|^\alpha \partial_s Q_{s/2} \kappa(x)\right)^2 dx\\
	& = 	\int_{\R}  \left(\Lgen [\F( \partial_s Q_{s/2} \kappa)](x)\right)^2 dx    =  \int_{\R}  \left(\Lgen [K_s \, m ](x)\right)^2 dx,
	\end{align*}
	by Plancherel's identity. Next, by the extended product rule, we have
	\begin{align*}
	\Lgen[fg]&=-D_\alpha[fg]-D^-_\alpha[fg] \\
	& = -fD_\alpha[g]-gD_\alpha[f]+\Lambda_\alpha[f,g]-fD^-_\alpha[g]-gD^-_\alpha[f]+\Lambda^-_\alpha[f,g]\\
	&= f\Lgen[g]+g\Lgen[f]+\Lambda_\alpha[f,g]+\Lambda^-_\alpha[f,g]
	\end{align*}
	for $f,g\in Dom(\Lgen)\subset Dom(D_\alpha)\cap Dom(D_\alpha^-)$. Hence
	\begin{align*}
	I_2 & \leq   4\, \int_{\R} K_s^2(x) \left(\Lgen [  m ](x)\right)^2 dx+ 4\, \int_{\R} m^2(x) \left(\Lgen [K_s ](x)\right)^2 dx  \\
	& \qquad + 4\, \int_{\R}  \left(\Lambda_\alpha [K_s , m ](x)\right)^2 dx + 4\, \int_{\R}  \left(\Lambda^-_\alpha [K_s , m ](x)\right)^2 dx.
	\end{align*}
	The last line above is bounded by a constant multiple of
	\begin{align*}
	& C_1^2 s^{2-2/\alpha} \int_{\R} K_1^2(x)  dx  + C_1^2 \int_{\R}  \left(\int_0^\infty |K_s(x-y)-K_s(x)| \frac{dy}{y^{1+\alpha}}\right)^2 dx \\
	&\qquad\qquad  \qquad+ C_1^2 \int_{\R}  \left(\int_0^\infty |K_s(x+y)-K_s(x)| \frac{dy}{y^{1+\alpha}}\right)^2 dx\\
	&\qquad\leq  C_1^2 s^{2-2/\alpha} \left( \|{K_1}\|_{L^2(\R)}^2  + 2 \|{J}\|_{L^2(\R)}^2 \right) \leq c\, C_1^2 s^{2-2/\alpha}.
	\end{align*}
\end{proof}

Then Theorem \ref{main_thm_3} below follows from the lines of the proof of the Theorem \ref{main_thm}.



\begin{theorem}\label{main_thm_3}
	Suppose  $\alpha \in (1/2,1)$ and that $m:\R\rightarrow\R$ is a bounded function with $m\in Dom(\Lgen)$, $$ \|m\|_\infty \leq C_1 \quad \mbox{and} \quad \left|  \Lgen[m](x) \right| \leq \frac{C_1}{|x|^\alpha} , \quad x\in \R-\{0\},$$ for some $C_1>0$. Then $m$ is a Fourier multiplier and the corresponding convolution operator $T_m$ can be extended from $L^p \cap L^2$ to $L^p$ for $p\in(1,\infty)$.
\end{theorem}


In the classical theory, Mikhlin multipliers have various applications. We believe that the new generalized class of multipliers defined in this paper will lead a broader range of applications and even restudy of old results.




 \bigskip \smallskip

 \it

 \noindent
   Department of Mathematics \\ I\c{s}\i k University\\
AMF233, \\34980 \c{S}ile, \\
Istanbul, TURKEY  \hfill Received: February 3, 2017 \\[4pt]
e-mail:  deniz.karli@gmail.com (primary); 

 \hspace{0,56cm} deniz.karli@isikun.edu.tr (secondary)

\end{document}